\documentclass[12pt]{amsart}

\usepackage[normalem]{ulem} 
\usepackage{soul}     
\usepackage{amsmath,amsthm, amsfonts, amssymb, amscd, mathrsfs}
\usepackage{graphicx, tikz-cd, subfig}
\usepackage{comment, fancyvrb, verbatim, lipsum}
\usepackage{xcolor, physics}
\usepackage{enumitem,multicol}
\usepackage[colorlinks=true, linkcolor=blue, citecolor=blue]{hyperref}
\theoremstyle{plain}
\newtheorem{Thm}{Theorem}
\newtheorem*{Thm*}{Theorem}
\newtheorem{Mainthminner}{Main Theorem}
\newenvironment{Mainthm}[1]{%
  \IfBlankTF{#1}
    {\renewcommand{\theMainthminner}{\unskip}}
    {\renewcommand\theMainthminner{#1}}%
  \Mainthminner
}{\endMainthminner}

\newtheorem{Manualthminner}{Theorem}
\newenvironment{Manualthm}[1]{%
  \IfBlankTF{#1}
    {\renewcommand{\theManualthminner}{\unskip}}
    {\renewcommand\theManualthminner{#1}}%
  \Manualthminner
}{\endManualthminner}

\newtheorem{Cor}{Corollary}

\newtheorem{Lem}{Lemma}
\newtheorem{Prop}{Proposition}

\theoremstyle{definition}
\newtheorem{Def}{Definition}
\newtheorem{Rem}{Remark}

\def\R{\mathbb{R}}

\def\S{\mathbb{S}}

\newcommand{\td}{\mathrm{d}}

\numberwithin{equation}{section}

\begin{document}

\title{Isotropic submanifolds of \texorpdfstring{$T\S^\MakeLowercase{n}$}{TS^n} and their focal sets}
\author{Nikos Georgiou}
\address{Nikos Georgiou\\
  Department of Mathematics\\
          South East Technological University\\
          Waterford\\
          Co. Waterford\\
          Ireland.}
\email{nikos.georgiou@setu.ie}

\author{Brendan Guilfoyle}
\address{Brendan Guilfoyle\\
          Faculty of Science and Informatics\\
          Munster Technological University, Kerry\\
          Tralee\\
          Co. Kerry\\
          Ireland.}
\email{brendan.guilfoyle@mtu.ie}

\author{Morgan Robson}
\address{Morgan Robson\\
  Department of Mathematics\\
          South East Technological University\\
          Waterford\\
          Co. Waterford\\
          Ireland.}
\email{morgan.robson@setu.ie}

\date{\today}

\begin{abstract}
Families of oriented lines in $\mathbb{R}^{n+1}$ are studied via their identification with submanifolds of $T\S^n$. In particular, families of oriented lines which are orthogonal to submanifolds in $\mathbb{R}^{n+1}$ are shown to characterise those which are isotropic with respect to the canonical symplectic structure on $T\S^n$.
   
Families of lines that are tangent to a $k$-dimensional submanifold of $\R^{n+1}$ are then studied. For such families, isotropy is shown to be equivalent to the generating vector field being geodesic and hypersurface-orthogonal on the submanifold.
   
The focal set in $\R^{n+1}$ of a family of lines is introduced, extending the classical definition for families normal to hypersurfaces, to general families of lines of arbitrary codimension. A formula is derived that expresses certain sectional curvatures of the focal set in terms of the signed distances between corresponding focal points. 

We then solve an inverse problem for the focal sets of hypersurfaces and show certain sectional and Ricci curvatures of the focal set are determined by the differences between the hypersurface's radii of curvature. This generalises a Theorem of Bianchi from 1874 - namely that surfaces in $\R^3$ of constant astigmatism have pseudo-spherical focal sets.
   
\end{abstract}
\maketitle
\section{Introduction and Statement of Results}
 The tangent bundle $T\S^n$ to the $n$-sphere possesses a well known symplectic structure $\Omega$, defined by pulling back the canonical symplectic structure on $T^\ast\S^n$ under the musical isomorphism $\flat_g:T\S^n\to T^\ast \S^n$, $g$ being the round metric on $\S^n$ \cite{GG22}. Furthermore, $T\S^n$ enjoys a natural identification with  $\mathbb{L}(\mathbb{R}^{n+1})$ - the space of oriented (affine) lines in $\R^{n+1}$. Various features of $\mathbb{L}(\R^{n+1})$ have been profitably studied through this identification \cite{Dun} \cite{Har16} \cite{S05}, most notably when $n=2$, where $\Omega$ is compatible with a pseudo-K\"ahler structure on $T\S^2$ with metric of signature $(2,2)$ \cite{GK}.
This in turn has been used to study differential-geometric problems in $\R^3$ \cite{GK2} \cite{GK3}. 
\\

An immersed submanifold $\Gamma\subseteq T\mathbb{S}^n$ is said to be \textit{isotropic} if $\Omega|_{\Gamma}=0$. The aim of this paper is to study isotropy in $T\S^n$ from the perspective of $\mathbb{L}(\R^{n+1})$. We first prove that $\Gamma$ is isotropic if and only if the lines are orthogonal to a submanifold of $\R^{n+1}$, at least locally (see Theorem \ref{Thm: structure theorem for orthogonal submanifolds}, Propositions \ref{prop: isotropic implies existance of orthogonal sub manifolds} and \ref{prop: orthogonal implies isotropic} - for a higher dimensional analogue see \cite{AB18}). Submanifolds of $T\S^n$ whose points (oriented lines) are defined through tangency to a submanifold of $\R^{n+1}$ are then studied. Formally, given a pair $(S,v)$ where $S\subset\R^{n+1}$ is an immersed submanifold and $v\in\Gamma(TS)$ a vector field of unit length, we define a map 
\begin{equation}\label{eq: map phi in intro}
\Phi_v:S\to T\S^n,
\end{equation}
which takes $p\in S$ to the oriented line passing through $p$ in the direction $v(p)$. Our first main theorem characterises isotropy for such images in $T\S^n$. Let $\langle\cdot,\cdot\rangle$ and $\overline{\nabla}$ denote the flat metric and connection on $\R^{n+1}$, respectively.
\begin{Mainthm}{1}\label{mainthm1}
Suppose $\Phi_v$ is an immersion and $\Phi_v(S)\subset T\S^n$ is immersed. The following are equivalent    \begin{enumerate}[label=(\alph*), itemsep=4pt, topsep=3pt]
        \item $\Phi_v(S)$ is isotropic.
        \item Each $p\in S$ has an open neighbourhood $U\subseteq S$ on which $v|_U=\mbox{grad~}f$ for some smooth function $f:U\to \mathbb{R}$. If $S$ is simply connected, $U=S$
        \item The distribution $v^\perp$ is integrable and $v$ is a geodesic vector field on $S$.
    \end{enumerate}
\end{Mainthm}

 We remark that condition $(c)$ is equivalent to $v$ being a normal vector field to a local foliation of $S$ by equidistant hypersurfaces. The above characterisation is then applied to study focal sets of isotropic families of lines. In this paper we define a subset of $\R^{n+1}$ called the \textit{focal set} of a family of oriented lines. In the case such lines are hypersurface-orthogonal, our definition (Definition \ref{def: focal points}) reduces (\mbox{Proposition \ref{Prop: orthgononal submanifolds distance from focal set is radii of curvature}}) to the classical definition, given as follows:

 For an oriented, strictly convex hypersurface $\Sigma\subset\R^{n+1}$ with unit normal vector field $\nu$, each radius of curvature $r_i:\Sigma\to \R$ ($i=1,\ldots,n$) defines a map $\mathcal{Y}_i:\Sigma\to\R^{n+1}$
\begin{equation} \label{eq: maps from surface to focal set}
    \mathcal{Y}_i(q):=q-r_i(q)\nu(q).
    \end{equation}
Classically, the image $\mathcal{Y}_i(\Sigma)$ is called the \textit{focal sheet of $\Sigma$ associated to $r_i$}. The union of these sheets for each radius of curvature defines the \textit{focal set} and the elements, the \textit{focal points} of $\Sigma$ \cite{CR78}. Consider the inverse problem:
\\ \hfill \\
\begin{minipage}{\linewidth}
\textbf{Problem}:~Given a hypersurface $S$, construct the hypersurfaces $\Sigma$ for which $S$ arises as a focal sheet.
\end{minipage}\\

If one fixes $S$, a hypersurface $\Sigma$ satisfying $S=\mathcal{Y}_i(\Sigma)$, induces a vector field over $S$ satisfying $(c)$ by Proposition \ref{prop: orthogonal hypersurface gives focal sets hypersurface folliation}. Conversely any vector field $v$ satisfying $(c)$ gives rise to an isotropic submanifold $\Phi_v(S) \subset T\S^n$ by Main Theorem \ref{mainthm1}. Under natural assumptions, e.g. $\Phi_v$ is injective and $\pi_1(S)=0$, a hypersurface $\Sigma\subset\R^{n+1}$ exists which is orthogonal to every oriented line of $\Phi_v(S)$, by Corollary \ref{cor: gamma simply connected and orthogonal surf} and Lemma \ref{lem: dphi is injective iff v is nowhere geodesic}. Then $S=\mathcal{Y}_i(\Sigma)$ for some $i\in\{1,\ldots,n\}$ by Proposition \ref{Prop: orthgononal submanifolds distance from focal set is radii of curvature}. The relation between a hypersurface and the focal set of its normal lines is a duality of orthogonality and tangency - a generalisation of the involute and evolute for plane curves, introduced by Huygens in 1673 when designing the optimal clock pendulum \cite{Arnold94} \cite{Arnold95} \cite{Huygens1673} \cite{Wilczynski16}.\\

Although an initial hypersurface $\Sigma$ may be smooth, the focal sheets $\mathcal{Y}_i(\Sigma)$ often have singularities. Understanding the types of singularities that can arise in this setting has been the subject of extensive study; see \cite{A90} \cite{BW91} \cite{BT23} \cite{Port71} and references therein. A source of singularities for the focal set are the zeros of the \textit{astigmatism}; $s_{ij}:\Sigma \to \R$
\begin{equation}
    s_{ij}(q):=r_i(q)-r_j(q).
\end{equation}
For $i\neq j$, $s_{ij}(q)=0$ iff $r_i(q)=r_j(q)$, thus at such points the radii of curvature have a multiplicity greater than $1$. At such points the functions $r_i$ become non-differentiable and the focal sheets lose regularity.
Here we work in the smooth scenario in which $\mathcal{Y}_i$ is a diffeomorphism and so $\mathcal{Y}_i(\Sigma)$ is a hypersurface. We investigate the curvature properties of such focal sets, a classical result in this direction being one of Bianchi (Theorem \ref{Bianchi thm}, 1874): Surfaces in $\R^3$ with constant astigmatism have focal sheets of constant and equal negative Gaussian curvature (i.e. are pseudo-spheres) \cite{B1874}.\\ Our next main theorem is a generalisation:
\begin{Mainthm}{2}\label{mainthm2}
Let $\Sigma\subset \R^{n+1}$ be a $C^3$-smooth, strictly convex hypersurface with unit normal field $\nu$ and radii of curvature $(r_i)_{i=1}^n$ of multiplicity $1$. For $i=1,2$ let $S_i:=\mathcal{Y}_i(\Sigma)$ be the focal sheet associated to $r_i$, and assume $\mathcal{Y}_i$ a diffeomorphism. Define $X_i\in\Gamma(T\Sigma)$ by
    \begin{align}\label{eq: normalised p vector}
       &r_i\overline{\nabla}_{X_i}\nu =X_i &&\text{and} &&\td r_i(X_i)=1,
    \end{align}
     so $X_i$ is a principal vector field of $\Sigma$. The sectional curvature of $S_1$, denoted $K_1$, satisfies
    \begin{equation*}
    (\mathcal{Y}_1^\ast K_1)\big(X_1,X_2)=\frac{\td s_{12}(X_1)-1}{(s_{12})^2}.
    \end{equation*}
\end{Mainthm}
Bianchi's Theorem follows by taking $n=2$ and the astigmatism of $\Sigma$ as constant (see \cite{LP20} for a classification of such $\Sigma$ in the rotationally symmetric setting). Note that above, $S_1$ and $S_2$ are arbitrary, with Main Theorem \ref{mainthm2} applying to any two distinct focal sheets of $\Sigma$.
Main Theorem \ref{mainthm2} thus asserts that a principal vector field $X$ of $\Sigma$, together with any other distinct principal vector field, induce a natural plane field on the focal sheet associated to $X$. The sectional curvature of the plane is determined entirely by the astigmatism of $\Sigma$. Summing over these planes gives:
\begin{Mainthm}{3}\label{mainthm3}
    Let $\Sigma$, $\nu$ and $(r_i)_{i=1}^n$ be as in Theorem \ref{mainthm2}. Let $S_j:=\mathcal{Y}_j(\Sigma)$ be focal sheets with $\mathcal{Y}_j$ a diffeomorphism for all $j\in\{1,\ldots,n\}$. If $X_j\in\Gamma(T\Sigma)$ satisfy equation (\ref{eq: normalised p vector}), then for $i\in\{1,\ldots,n\}$ the Ricci curvature of $S_i$ satisfies
    \[
    (\mathcal{Y}_i^\ast\mathrm{Ric})(X_i,X_i)=\sum_{\substack{j=1\\j\neq i}}^n\frac{X_i(s_{ij})-1}{(s_{ij})^2}.
    \]
\end{Mainthm}

In fact, the above constructions are not limited to codimension $1$ and work on any embedded submanifold subject to appropriate conditions (see Remark \ref{rem: Thm 2 with arb codim}).\\

Finally, we note that the relationship between two focal sheets of a hypersurface may be viewed through the lens of integrable systems. A \textit{geometric Bianchi-B\"acklund transformation} (BBT) is (broadly) a translation along a family of lines tangent to two submanifolds of $\R^{n+1}$. When such lines are viewed as tangent vector fields (Cf. the map (\ref{eq: map phi in intro})), \mbox{Proposition \ref{prop: submanfolds tangent to the lines are focal}} shows that a BBT defines a map from the focal set of its line family to itself. Given a hypersurface $\Sigma$ with astigmatisms $(s_{ij})_{i,j=1}^n$ and satisfying the conditions of Main Theorem \ref{mainthm2}, the map
\begin{equation}\label{equ: map between focal}
\mathcal{Y}_j\circ\mathcal{Y}_i^{-1}: \mathcal{Y}_i(\Sigma) \to \mathcal{Y}_j(\Sigma),
\end{equation}
is a BBT, translating along lines tangent to both focal sheets by a distance $|s_{ij}|$.\\

More narrowly, the classical BBT maps between pseudo-spherical surfaces in $\R^{3}$, and moves points a constant distance along common tangent lines \cite{RaS82}. Numerous generalisations of the classical BBT exist, including pseudo-spherical surfaces in higher codimension \cite{Gorkavvy2012}\cite{Gorkavvy2015} and $n$-dimensional hyperbolic spaces in ${\mathbb R}^{2n-1}$ \cite{Aminov1978} \cite{Aminov2011} \cite{TaT79} \cite{Terng1980}.
If $\Sigma$ is of constant astigmatism, the map (\ref{equ: map between focal}) is a generalisation of the classical BBT. In our generalisation the family of lines must be isotropic and Main Theorem \ref{mainthm2} implies that the focal sheets only contain certain planes of negative sectional curvature, rather than being pseudo-spherical.\\ 

The paper is organised as follows. Section \ref{sec:preliminaries} fixes notation and describes in more detail the symplectic structure $(T\S^n,\Omega)$, as well as the identification of $T\S^n$ with $\mathbb{L}(\R^{n+1})$. The focal set of a submanifold $\Gamma\subset T\S^n$ is defined in Section \ref{sec:focal sets} and its properties described. In Section $\ref{sec: lines normal to submanifolds}$ we show that families of lines are isotropic iff they are submanifold-orthogonal (Propositions \ref{prop: isotropic implies existance of orthogonal sub manifolds} and \ref{prop: orthogonal implies isotropic}) and a structure theorem for such orthogonal submanifolds is proven (Theorem \ref{Thm: structure theorem for orthogonal submanifolds}). Futhermore, lines orthogonal to hypersurfaces are shown to have focal sets which coincide with the classical definition (\mbox{Proposition \ref{Prop: orthgononal submanifolds distance from focal set is radii of curvature}}). Section \ref{sec: lines tangent submanifolds} proves Main Theorem \ref{mainthm1}. In Section \ref{sec: The Curvature of Focal Sets}, expressions for the sectional and Ricci curvature of focal sets (when they are $C^2$-smooth) of isotropic line families are given in term of the distance between focal points (Theorems \ref{Thm: sectional curvature from mu} and \ref{Thm: curvature control from multiple focal set componants}). Finally in Section \ref{sec: the focal sets of hypersurfaces} we specialise to the focal sets of hypersurfaces in $\R^{n+1}$, wherein Main Theorem \ref{mainthm2} and Main Theorem \ref{mainthm3} are proven.

\section{Preliminaries}\label{sec:preliminaries}
\subsection{The Symplectic Form and Splitting of \texorpdfstring{$TT\S^n$}{TTSn}} \hfill \\
Let $g$ be the round metric on $\S^n$ and $D$ the Levi-Civita connection. $D$ induces a splitting of the double tangent bundle $TT\S^n=H\S^n\oplus V\S^n$ with $V\S^n$ and $H\S^n$ the \textit{vertical} and the \textit{horizontal} subbundles of $TT\S^n$, respectively. They are defined as
\begin{align*}
    &V\S^n:=\ker (\td\pi), && &&H\S^n:=\ker(\mathcal{K}),
\end{align*}
with $\pi:T\S^n\to \S^n$ being the canonical projection map and $\mathcal{K}:TT\S^n\to T\S^n$ being the \textit{connection map} defined as follows: Given $(\mu_0,X)\in TT\S^n$, let $\mu(t)$ be a curve in $T\S^n$ such that $\mu(0)=\mu_0$ and $\mu'(0)=X$ and denote by $\frac{\mathrm{D}}{\td t}$ the covariant derivative with respect to $D$ along the curve $(\pi\circ\mu)(t)$ in $\S^n$. Now define
\[
\mathcal{K}(\mu_0,X):=\left.\frac{\mathrm{D}\mu(t)}{\td t}\right|_{t=0}.
\]
The restrictions $\td \pi|_{H\S^n}:H\S^n\to T\S^n$ and $\mathcal{K}|_{V\S^n}:V\S^n\to T\S^n$ are isomorphisms and induce the isomorphism $TT\S^n =H\S^n\oplus V\S^n \simeq T\S^n\oplus T\S^n$ via  $X\simeq(\td\pi(X), \mathcal{K}(X))$. See \cite{Dom} and \cite{Kow} for further details. Additionally, $T\S^n$ supports a symplectic structure $\Omega:=-\td \theta$, with $\theta:TT\S^n\to\mathbb{R}$ the associated Liouville $1$-form given by
\begin{equation}\label{eq:louiville form in terms of round metric}
    \theta_{(p,\beta)}(X)=g_p(\td \pi_{(p,\beta)}(X),\beta),
  \end{equation}
where $(p,\beta)\in T\S^n$ and $X\in T_{(p,\beta)}T\S^n$. Upon taking an exterior derivative one finds
\begin{equation}\label{e:symplectic_structure}
\Omega(X,Y)=g(\td\pi(X),\mathcal{K}(Y))-g(\td\pi(Y),\mathcal{K}(X)),
\end{equation}
where for brevity we have omitted the base point $(p,\beta)$. An equivalent definition of $\theta$ can be given by pulling back the tautological $1$-form on $T^\ast \S^n$ under the musical isomorphism $\flat:T\S^n\to T^*\S^n$ induced by $g$. A coordinate description in any Riemannian manifold, see \cite{AGK11} \cite{AGK11a}.
 \subsection{Oriented Lines and \texorpdfstring{$T\S^n$}{TSn}}\hfill \\
 Denote by $\langle \cdot, \cdot \rangle$ the Euclidean inner product and by $\overline{\nabla}$ the flat connection on $\mathbb{R}^{n+1}$. A given point $(\xi,\eta)\in T\S^n\subseteq\R^{n+1}\times \R^{n+1}$ can be identified with the oriented line $\{\eta+t\xi\,:\, t\in {\mathbb R}\}$, so that $\xi$ gives the direction of the line and $\eta$ the closest point on the line to the origin. Given a parameter $r\in \mathbb{R}$ and an oriented line $(\xi,\eta)$, the map
 \begin{equation}\label{eq:definition_of_psi}
\begin{aligned}
    \Psi:&~T\mathbb{S}^n\times \mathbb{R}\to\mathbb{R}^{n+1}\\
    &((\xi, \eta),r) \longmapsto \eta+r\xi,
\end{aligned}
\end{equation}
returns the point in $\mathbb{R}^{n+1}$ lying a distance of $r$ along the line from $\eta\in\mathbb{R}^{n+1}$. The image $\Psi(\{(\xi,\eta)\}\times\mathbb{R})$ is thus the set of points in $\mathbb{R}^{n+1}$ which lie on the oriented line $(\xi,\eta)$. More generally for a submanifold $\Gamma\subseteq T\mathbb{S}^n$ the set $\Psi(\Gamma\times \mathbb{R})$ is the set of all points of $\mathbb{R}^{n+1}$, which lie on some line in $\Gamma$.

Conversely, given an  oriented line $\{p+tv\,:\, t\in {\mathbb R}\}$, for $p,v\in \mathbb{R}^{n+1}$ and $|v|=1${,} we can identify this line with the pair $(v,p-\langle p,v\rangle v)$ regarded as an element of $T\S^n$, due to $v$ being unit and orthogonal to $v-\langle p, v\rangle v$. Thus submanifolds of $T\S^n$ correspond to families of lines in $\mathbb{R}^{n+1}$ and vice versa. Suppose  $S\subseteq\mathbb{R}^{n+1}$ is a smooth immersed submanifold and $v$ a smooth section of $T\mathbb{R}^{n+1}$ over $S$
\[
v: p\in S \longmapsto v(p) \in T_p\mathbb{R}^{n+1},
\]
such that $||v(p)||=1$. We will think of the pair $(S,v)$ as generating a family of oriented lines, with $v$ assigning to each $p\in S$ the oriented line which passes through $p$ in the direction $v(p)$. This assignment is done through the smooth mapping
\begin{equation}\label{equ: definition of Phi}
\begin{aligned}
\Phi_v:&\,S\rightarrow T{\mathbb S}^n  \\
&p\longmapsto (v(p),p-\left<v(p),p\right>v(p)).
\end{aligned}
\end{equation}
For a general choice of $S$ and $v$, the subset $\Phi_v(S)\subset T\S^n$ may not have the structure of a submanifold, even when the map $\Phi_v$ is an immersion. We will however frequently assume this is the case. 
\subsection{Notation and Conventions}
\begin{enumerate}[label=\arabic*.]
    \item When we have the need to consider other sections of $T\R^{n+1}$ over $S$, say $X$ or $Y$, we will use the notation $\Phi_X$ or $\Phi_{Y}$ etc. Often $\nu$ will denote a normal section, $\nu\in\Gamma(TS)$, and $v$ will denote either an arbitrary or tangent section.
    \item There may be topological restrictions on $S$ for $v$ to be defined globally, specifically when $v\in\Gamma(TS)$. Throughout this work we assume the pair $(S,v)$ are sufficient for this to be the case. 
    \item Given a smooth manifold $\mathcal{M}$ a subset $\mathcal{N}\subseteq\mathcal{M}$ is called immersed if $\mathcal{N}$ itself is a smooth manifold and the inclusion map $\mathcal{N}\hookrightarrow\mathcal{M}$ is a smooth immersion.
    \item As usual we will identify $\mathbb{R}^n$ with its tangent spaces, permitting us to add together elements in different fibres of $T\mathbb{R}^{n}$.
    \item In this paper `smooth' will be taken to be $C^{2}$, unless otherwise stated.
\end{enumerate}
\section{Focal Points of Submanifolds in \texorpdfstring{$T\S^n$}{TSn}}
\label{sec:focal sets}
We first describe the derivative $\td\Psi$ of the map $\Psi$, defined in (\ref{eq:definition_of_psi}), in terms of the maps $\td \pi$, $\mathcal{K}$ and the Liouville form $\theta$. 

\begin{Lem}\label{Lem: decomposition of psi}
     For any $Y\in T_{((\xi_0,\eta_0),r_0)}(T\S^n\times \mathbb{R})\simeq  T_{(\xi_0,\eta_0)}T\S^n \oplus T_{r_0}\mathbb{R}$ write $Y=X+\rho\displaystyle{\pdv{}{r}}$ with $X\in T_{(\xi_0,\eta_0)}T\S^n$, $\rho\in\mathbb{R}$, and $r$ the canonical coordinate on $\mathbb{R}$. Then
     \begin{equation}\label{eq: decomposition of dPsi}
     \td\Psi_{((\xi_0,\eta_0),r_0)}(Y)=r_0\td\pi(X)+\mathcal{K}(X)+(\rho-\theta(X))\xi_0.
     \end{equation}
 \end{Lem}
 \begin{proof}
      If $\gamma_{(\eta,\xi)}$ is the geodesic passing through $\eta$ in the direction $\xi$, then we may write $\Psi((\xi,\eta),r)=\gamma_{(\eta,\xi)}(r)$. Hence if $\sigma(t)=(\xi(t),\eta(t))$ is a curve such that $X=\sigma'(0)$,
    \begin{align*}
        \td\Psi_{((\xi_0,\eta_0),r_0)}(Y)&=\td\Psi_{((\xi_0,\eta_0),r_0)}(X)+\rho\td\Psi_{((\xi_0,\eta_0),r_0)}\left(\displaystyle\frac{\partial}{\partial r}\right)\\
        &= \left.\dv{}{t}\right|_{t=0}\gamma_{(\eta(t),\xi(t))}(r_0)+\rho \gamma_{(\eta_0,\xi_0)}{}'(r_0).
    \end{align*}
    Hence $\td\Psi_{((\xi_0,\eta_0),r_0)}(Y)=J(r_0)$ where $J(r)$ is the Jacobi field along $(\xi_0,\eta_0)$ satisfying
    \begin{align*}
        &J(0)=\eta'(0)+\rho \xi_0, &&\nabla_rJ (0)=\xi'(0).
    \end{align*}
     Since, $\xi(t)\cdot\eta(t)\equiv0$, the connection $D$ on $\S^n$ is tangential, and by equation (\ref{eq:louiville form in terms of round metric});
     \[
     \mathcal{K}(X)=\left.\frac{D\eta(t)}{dt}\right|_{t=0}=\eta'(0)-\langle\eta'(0),\xi_0\rangle\xi_0=\eta'(0)+\langle\eta_0,\xi'(0)\rangle\xi_0=\eta'(0)+\theta(X)\xi_0.
     \]
     Hence the initial conditions may be expressed as
     \begin{align*}
        &J(0)=\mathcal{K}(X)+(\rho-\theta(X))\xi_0, &&\nabla_rJ (0)=\td\pi(X).
    \end{align*}
     The claim now follows by integrating the Jacobi equation for $J(r)$.
 \end{proof}
 \begin{Def}\label{def: focal points}
    A \emph{focal point} of a submanifold $\Gamma\subseteq T\mathbb{S}^n$ is a point $p\in\Psi(\Gamma\times \mathbb{R})$ such that $\Psi|_{\Gamma\times\mathbb{R}}$ fails to be an immersion at some point in $\Psi^{-1}(p)$. The union over such $p$ is called the \emph{focal set of} $\Gamma$.
\end{Def}
For a Riemannian manifold $(\mathcal{M},\mathbb{G})$ the notion of a focal point of a submanifold also exists, defined in terms of Jacobi fields \cite{DoCarmo92}. Hence a submanifold of $(T\S^n,\mathbb{G})$, for some metric $\mathbb{G}$, may have focal points in this sense. We remark that Definition \ref{def: focal points} is distinct from this notion.
\begin{Prop} \label{prop: characterisation of focal points}
    Let $((\xi,\eta),r)\in \Gamma \times \mathbb{R}$. Then $p=\Psi((\xi,\eta),r)$ is a focal point of $\Gamma$ if and only if there exists a non-zero $X\in T_{(\xi,\eta)}\Gamma$ such that $r\td\pi(X)+\mathcal{K}(X)=0$.
\end{Prop}
\begin{proof}
$p$ is focal iff there exists non-zero $Y\in T_{((\xi,\eta),r)}(\Gamma\times \R)$ with $\td\Psi_{((\xi,\eta),r)}(Y)=0$. Write $Y=X+\rho\pdv{}{r}$. Since $\td\pi(X),\mathcal{K}(X)\perp\xi$, Lemma \ref{Lem: decomposition of psi} implies $p$ is focal iff there exists $X\in T_{(\xi,\eta)}\Gamma$ and $\rho\in \mathbb{R}$ such that $r\td\pi(X)+\mathcal{K}(X)=0$ and $\rho=\theta(X)$. 
\end{proof}
There is a more geometric characterisation of a focal point. Denote by $\mathcal{L}(c)$ the set of oriented lines passing through a fixed point $c\in\mathbb{R}^{n+1}$. It is not hard to see that
\[
\mathcal{L}(c)=\{(\xi,c-\langle c,\xi\rangle\xi)~:~\xi\in\S^n\}.
\]
The following lemma shows that $c\in {\mathbb R}^{n+1}$ is the only focal point of $\mathcal{L}(c)$.
\begin{Lem}\label{Lem: tangent space of lines through a point}
    Suppose $(\xi,\eta)\in \mathcal{L}(c)$ for some $c\in\mathbb{R}^{n+1}$. Then
    \[
    T_{(\xi,\eta)}\mathcal{L}(c)=\{X\in T_{(\xi,\eta)}T\S^n: \langle c,\xi\rangle\td\pi(X)+\mathcal{K}(X)=0\}.
    \]
\end{Lem} \begin{proof}
We proceed by double inclusion. Fix $c\in\mathbb{R}^{n+1}$. For all $(\xi,\eta)\in\mathcal{L}(c)$ we have
\[
c=\Psi((\xi,\eta),\langle c,\xi\rangle).
\]
Differentiating with respect to $X\in T_{(\xi,\eta)}\mathcal{L}(c)$ gives $\langle c, \xi \rangle\td\pi(X)+\mathcal{K}(X)=0$ (by Lemma \ref{Lem: decomposition of psi} and $\td\pi(X),\mathcal{K}(X)\perp\xi$), implying the $\subseteq$ direction. To show the $\supseteq$ direction, suppose $X\in T_{(\xi,\eta)}T\S^n$ and $\langle c,\xi\rangle\td\pi(X)+\mathcal{K}(X)=0$. Define the curve $\sigma$ in $\mathcal{L}(c)$ by
\[
\sigma(t)=\big(\xi(t),c-\langle c,\xi(t)\rangle\xi(t)\big),
\]
such that $\sigma(0)=(\xi,\eta)$ and $\xi'(0)=\td\pi(X)$. Note that $\td\pi(\sigma'(0))=\td\pi(X)$ also. Furthermore since $\sigma(t)$ is a curve in $\mathcal{L}(c)$, the already proved $\subseteq$ direction implies
\[
\mathcal{K}(\sigma'(0))=-\langle c,\xi\rangle\td\pi(\sigma'(0))=-\langle c, \xi \rangle \td \pi(X)=\mathcal{K}(X).
\]
Hence as their horizontal and vertical components agree, $X=\sigma'(0)\in T_{(\xi,\eta)}\mathcal{L}(c)$.
\end{proof}
\begin{Cor}
   Let $(
    (\xi,\eta),r)\in \Gamma \times \mathbb{R}$. Then $p=\Psi((\xi,\eta),r)$ is a focal point of $\Gamma$ if and only if $T_{(\xi,\eta)}\Gamma~\cap~T_{(\xi,\eta)}\mathcal{L}(p)\neq \{0\}$.
\end{Cor}
\begin{Prop}
    Let $(\xi,\eta)\in \Gamma$. If $\dim \Gamma > n$, every point along the oriented line $(\xi,\eta)$ is focal, or if $\dim \Gamma \leq n$ then there are at most $n$ focal points along $(\xi,\eta)$.
\end{Prop}
\begin{proof}
    Fix $(\xi,\eta)\in\Gamma$ and let $c_r=\Psi((\xi,\eta),r)$, with $r$ a real variable. By Lemma \ref{Lem: tangent space of lines through a point} and since $\td\pi|_{H\S^n}:H\S^n\to T\S^n$ and $\mathcal{K}|_{V\S^n}:V\S^n\to T\S^n$ are linear isomorphisms, bases of $H_{(\xi,\eta)}\mathbb{S}^n,V_{(\xi,\eta)}\mathbb{S}^n$ and $T_\xi \S^n$ can be chosen as to make the identifications
    \begin{align*}
        & H_{(\xi,\eta)}\mathbb{S}^n \oplus V_{(\xi,\eta)}\mathbb{S}^n = \mathbb{R}^{n}\oplus\mathbb{R}^{n}, 
        &&T_{(\xi,\eta)}\Gamma = \ker A, &&T_{(\xi,\eta)}\mathcal{L}(c_r) = \mbox{Im }G(r),
    \end{align*}
    where $A$ and $G(r)=[I_n | -r I_n]^T$ are respectively $(2n-\dim \Gamma) \times 2n$ and $2n\times n$ matrices with $I_n$ the identity matrix. $AG(r)$ is thus a  $(2n-\dim \Gamma) \times n$ matrix and
    \[
    T_{(\xi,\eta)}\Gamma \cap T_{(\xi,\eta)}\mathcal{L}(c_r) \neq \{0\} \iff \ker AG(r) \neq \{0\} \iff \mbox{rank }AG(r) < n.
    \]
    If $\dim  \Gamma >n$, then $\mbox{rank }AG(r)<n$ due to the size of $AG(r)$, so the line ${(\xi,\eta)}$ possesses a focal point (namely $c_r$) for any value of $r$. Conversely if $\dim \Gamma \leq n$ then 
    \[
    \mbox{rank }AG(r) < n \iff \text{all } n \times n \text{ minors of } AG(r) \text{ vanish.}
    \]
    Each minor is a polynomial in $r$ of at most degree $n$ so has at most $n$ real roots.
\end{proof}
\begin{Rem}
    A focal point exists on a given line iff
    $\binom{2n-{\dim \Gamma}}{n}$ 
     polynomials (the number of minors in the above proof) have a common solution. This number is decreasing in $\dim \Gamma$ and for $\dim \Gamma=n$ results in a single polynomial equation for $r$. Hence focal points become less common the smaller $\dim \Gamma$ is relative to $n$.
\end{Rem}
\section{Lines Normal to Submanifolds}
\label{sec: lines normal to submanifolds}
\begin{Def}\label{def: orthogonal surface}
    Let $\Gamma\subseteq T\S^n$ and $\Sigma\subset\mathbb{R}^{n+1}$ be immersed submanifolds. $\Sigma$ is called an \emph{orthogonal submanifold of} $\Gamma$ if $\dim \Sigma>0$ and there exists a smooth unit normal section $\nu\in\Gamma (N\Sigma)$ such that $\Phi_\nu(\Sigma)\subseteq\Gamma$, where $\Phi_\nu$ is given in (\ref{equ: definition of Phi}).
\end{Def}
The next proposition shows that when $\Gamma$ is isotropic, dimension $\dim \Gamma$ orthogonal submanifolds of $\Gamma$ exist and pass through every point of $\Psi(\Gamma\times\mathbb{R})$ which is not focal.
 \begin{Prop} \label{prop: isotropic implies existance of orthogonal sub manifolds}
Let $\Gamma\subseteq T\mathbb{S}^n$ be isotropic and $\dim \Gamma>0$. Suppose that $p\in\Psi(\Gamma\times\mathbb{R})$ is not a focal point. Then $p$ is contained in an orthogonal submanifold of $\Gamma$ of dimension $\dim \Gamma$ which takes the form
\[
\Sigma_0:=\Psi\Big(\Big\{((\xi,\eta),F(\xi,\eta)): (\xi,\eta)\in W\Big\}\Big),
\]
with $W\subset \Gamma$ open and $F:W\to \mathbb{R}$ a smooth function such that $\theta=\td F$ and $F(\xi_0,\eta_0)=r_0$ where $p=\Psi((\xi_0,\eta_0),r_0)$ and $\theta$ is the Liouville 1-form (\ref{eq:louiville form in terms of round metric}).
\end{Prop}
\begin{proof}
Let $p=\Psi((\xi_0,\eta_0),r_0)$. Since $\Gamma$ is isotropic, $\theta|_{\Gamma}$ is locally exact, so there exists a neighbourhood $W\subseteq\Gamma$ of $(\xi_0,\eta_0)$ and a smooth function $F:W\to\mathbb{R}$ with $\theta=\td F$. Without loss of generality assume $F(\xi_0,\eta_0)=r_0$. Define $\Sigma_0=\Psi(\text{graph}(F))$ where 
\begin{align*}
 \mbox{graph}(F)=\Big\{\big((\xi,\eta),F(\xi,\eta)\big):(\xi,\eta)\in W\Big\}.
\end{align*}

As $p$ is not focal, $\Psi$ is a local diffeomorphism near $((\xi_0,\eta_0),F(\xi_0,\eta_0))$ hence $\Psi|_{\text{graph}(F)}$ is a diffeomorphism onto $S_0$ (after possibly shrinking $W$). Thus $\Sigma_0$ is immersed and $\dim \Sigma_0=\dim\Gamma$. To show $\Sigma_0$ satisfies Definition \ref{def: orthogonal surface}, if $q\in \Sigma_0$, there exists a unique $(\xi_q,\eta_q)$ such that $q=\Psi(\xi_q,\eta_q,F(\xi_q,\eta_q))$, and the map $q\mapsto\xi_q$ is smooth. Define the unit section $\nu:q\mapsto\xi_q$. It follows that $\Phi_\nu(\Sigma_0)=W\subseteq\Gamma$. Also $\xi_q$ is normal to $T_q\Sigma_0$;
\begin{align*}
T_q\Sigma_0&=\td\Psi(T_{(\xi_q,\eta_q,F(\xi_q,\eta_q))}\text{graph}(F))\\
&=\{F(\xi_q,\eta_q)\td\pi(X)+\mathcal{K}(X) + (\td F(X)-\theta(X))\xi_q : X\in T_{(\xi_q,\eta_q)}\Gamma\}\\
&=\{F(\xi_q,\eta_q)\td\pi(X)+\mathcal{K}(X) : X\in T_{(\xi_q,\eta_q)}\Gamma\},
\end{align*}
the second line following by Lemma \ref{Lem: decomposition of psi}. Hence $\xi_q\in N_q\Sigma_0$ and thus $\nu\in\Gamma(N\Sigma_0)$.
\end{proof}
\begin{Cor}\label{cor: gamma simply connected and orthogonal surf}
    If $\Gamma$ is assumed simply connected, we may take $W=\Gamma$ so that a submanifold of $\R^{n+1}$ exists to which $\Gamma$ is globally orthogonal.
\end{Cor}
Among orthogonal submanifolds of $\Gamma$, those such as $\Sigma_0$ are `maximal' since other orthogonal submanifolds are contained in them - at least locally.
\begin{Thm}[Structure Theorem for Orthogonal Submanifolds]\label{Thm: structure theorem for orthogonal submanifolds}
    Let $\Gamma\subset T\S^n$ be isotropic and $\mathcal{U}$ an orthogonal submanifold of $\Gamma$ without focal points. For all $p\in \mathcal{U}$ there exists $V_p$, an open neighbourhood of $p$ in $\mathcal{U}$ such that $V_p\subset \Sigma_p$, where
    \[
    \Sigma_p:=\Psi\Big(\Big\{((\xi,\eta),F(\xi,\eta)): (\xi,\eta)\in W\Big\}\Big),
    \]
    with $W\subseteq\Gamma$ open in $\Gamma$ and $F:W\to \mathbb{R}$ a smooth function such that $\theta|_W=\td F$.
\end{Thm}
Before proving this we first prove a lemma concerning the map (\ref{equ: definition of Phi}).
\begin{Lem}\label{Lem: derivative of Phi_v}
    Let $S\subset \R^{n+1}$ be immersed and $v$ a unit section over $S$. For all $X\in T_pS$ the map $\td\Phi_v:TS\to TT\S^n$ satisfies
    \begin{equation}\label{eq: derivative of Phi_v}
    \begin{aligned}
        &\td\pi(\td\Phi_v(X))=\overline\nabla_Xv(p),\\ &\mathcal{K}(\td\Phi_v(X))=X-\langle X,v(p)\rangle v(p)-\langle p, v(p)\rangle \overline\nabla_X v(p),
    \end{aligned}
    \end{equation}
    and $\ker \td\Phi_v|_p \subseteq \mbox{Span}\{v(p)\}$. Furthermore if $v(p)\notin T_pS$, then $\td\Phi_v|_p$ is injective.
\end{Lem}
\begin{proof}
    Let $\alpha(t)$ be a curve in $S$ such that $\alpha'(0)=X$. Then  
\[
(\Phi_v\circ \alpha)(t)=\big((v\circ\alpha)(t),\alpha(t)-\langle\alpha(t),(v\circ\alpha)(t)\rangle (v\circ\alpha)(t)\big),
\]
is a curve in $T\mathbb{S}^n$ such that $(\pi\circ\Phi_v\circ\alpha)(t)=(v\circ\alpha)(t)$, and $\td\pi(\td\Phi_v(X))=\overline\nabla_Xv(p)$ follows after differentiating. The second expression follows from taking the covariant derivative of $\alpha(t)-\langle\alpha(t),(v\circ\alpha)(t)\rangle (v\circ\alpha)(t)$ with respect to $D$ along the curve $(v\circ\alpha)(t)$ in $\S^n$.  If `$\top$' denotes the projection of $T_{v(p)}\mathbb{R}^{n+1}$ onto $T_{v(p)}\S^n$;
\begin{align*}
    \mathcal{K}(\td\Phi_v(X))&=\left.\frac{D}{\td t}\left(\alpha(t)-\langle\alpha(t),(v\circ\alpha)(t)\rangle (v\circ\alpha)(t)\right)\right|_{t=0}\\
    &=\left.\left(\dv{}{t}\Big[ \alpha(t)-\langle\alpha(t),(v\circ\alpha)(t)\rangle (v\circ\alpha)(t)\Big]\right|_{t=0}\right)^\top\\
    &=\left(X-\langle X,v(p)\rangle-\langle p,\overline\nabla_Xv(p)\rangle v(p)-\langle p,v(p)\rangle\overline\nabla_Xv(p)\right)^\top\\
    &=X-\langle X,v(p)\rangle v(p)-\langle p, v(p)\rangle \overline\nabla_X v(p).
\end{align*}
To prove the last claims, let $X\in\ker \td\Phi|_p$. Then $\td\Phi(X)=0$ implying $\td\pi(\td\Phi_v(X))=\mathcal{K}(\td\Phi_v(X))=0$. This implies $X=\langle X,v(p)\rangle v(p)$ from the first part of the lemma.
\end{proof}
\begin{proof}[Proof of Theorem \ref{Thm: structure theorem for orthogonal submanifolds}]
    Let $\mathcal{U}$ be an orthogonal submanifold of $\Gamma$ and let $\nu\in\Gamma(N\mathcal{U})$ be unit such that $\Phi_\nu(\mathcal{U})\subseteq \Gamma$. Fix $p\in \mathcal{U}$. Since $\nu(p)\in N_p\mathcal{U}$, the last part of Lemma \ref{Lem: derivative of Phi_v} implies that for some open neighbourhood $V_p\subset \mathcal{U}$ of $p$,  $\Phi_\nu|_{V}$ is a diffeomorphism onto its image.  Define the smooth function
    \begin{align*}
        &r:~\mathcal{U}\to\mathbb{R}, &&q \mapsto \langle q, \nu(q)\rangle.
    \end{align*}
   By the definitions of $\Psi$ and $\Phi_\nu$ one sees that
    $q=\Psi(\Phi_\nu(q),r(q))$ for all $q\in \mathcal{U}$. Hence
    \[
    V_p=\Psi\Big(\Big\{(\Phi_\nu(q),r(q)): q\in V_p\Big\}\Big)=\Psi\Big(\Big\{\big((\xi,\eta),(r\circ\Phi_\nu^{-1})(\xi,\eta)\big): (\xi,\eta)\in \Phi_\nu(V_p)\Big\}\Big).
    \]
    On the other hand, since $\Gamma$ is isotropic and $p$ is not focal, Proposition \ref{prop: isotropic implies existance of orthogonal sub manifolds},
    \[
    \Sigma_p:=\Psi\Big(\Big\{((\xi,\eta),F(\xi,\eta)): (\xi,\eta)\in W\Big\}\Big),
    \]
    is a dimension $\dim \Gamma$ orthogonal submanifold of $\Gamma$ containing $p$, with $\Phi_\nu(p)\in W\subset \Gamma$ open and $F:W\to\mathbb{R}$ a smooth function such that $\td F=\theta$. By shrinking $V_p$ if required we may assume that $\Phi_\nu(V_p)\subset W$. Note further that if $X\in T_q\mathcal{U}$, then
    \begin{align*}
        \td r(X)=X(r)&=\langle X, \nu(q) \rangle + \langle q, \overline\nabla_X \nu(q)\rangle,\\
        &=\langle q, \overline\nabla_X \nu(q)\rangle, &&\text{(Since $X\perp \nu(p)$)}\\
        &=\langle q-\langle q,v(q)\rangle v(q), \overline\nabla_X v(q)\rangle, &&\text{(Since $\nabla_X\nu(p)\perp \nu(p)$)}\\
        &=\langle q-\langle q,v(q)\rangle v(q), \td\pi(\td\Phi_\nu(X))\rangle, &&\text{(By Lemma \ref{Lem: derivative of Phi_v})}\\
        &=\Phi_\nu^\ast\theta(X), &&\text{(From equation (\ref{eq:louiville form in terms of round metric}))}\\
        &=\td(\Phi_\nu^\ast F)(X). &&\text{(By definition of $F$)}
    \end{align*}
    Hence $r=\Phi_\nu^\ast (F|_{\Phi_\nu(V)})$, since we may choose $F$ to satisfy $F(\Phi_\nu(p))=r(p)$. Thus
    \begin{align*}
        V_p=\Psi\Big(\Big\{((\xi,\eta),F(\xi,\eta)):(\xi,\eta)\in \Phi_\nu(V_p)\Big\}\Big)
        \subseteq \Psi\Big(\Big\{((\xi,\eta),F(\xi,\eta)): (\xi,\eta)\in W\Big\}\Big)
        =\Sigma_0.
    \end{align*}
\end{proof}
We now show that the existence of orthogonal submanifolds implies isotropic.
\begin{Lem} \label{Lem: pullback of omega} Let $S\subset\R^{n+1}$, $\Phi_v(S)\subset T\S^n$ be immersed with $v$ a unit section over $S$. For $X,Y\in T_pS$
\[
(\Phi_v^\ast\Omega)_p(X,Y)=\langle \overline\nabla_X v(p), Y \rangle-\langle \overline\nabla_Y v(p), X \rangle.
\]
\end{Lem}
\begin{proof}
If $p\in S$ and $X,Y\in T_pS$ then equation (\ref{e:symplectic_structure}) gives
    \begin{align*}
    (\Phi_v^\ast \Omega)(X,Y)=g\Big(\td\pi(\td\Phi_v(X)),\mathcal{K}(\td\Phi_v(Y))\Big)-g\Big(\td\pi(\td\Phi_v(Y)),\mathcal{K}(\td\Phi_v(X))\big).
    \end{align*}
    Using Lemma \ref{Lem: derivative of Phi_v} gives
    \begin{align*}
    g\Big(\td\pi(\td\Phi_v(X)),\mathcal{K}(\td\Phi_v(Y))\Big)&=g\Big(\overline\nabla_Xv(p), Y-\langle v(p),Y\rangle v(p)-\langle v(p),p\rangle \overline\nabla_Y v(p) \Big)\\&=\langle \overline\nabla_X v(p), Y \rangle-\langle v(p),p\rangle \langle \overline\nabla_X v(p),\overline\nabla_Y v(p) \rangle.
    \end{align*}
    A similar calculation with $X$ and $Y$ exchanged leads to the formula
    \begin{align*}
    (\Phi_v^\ast \Omega)_p(X,Y)&=\langle \overline\nabla_X v(p), Y \rangle-\langle \overline\nabla_Y v(p), X \rangle.
    \end{align*}
\end{proof}
For the rest of this section we use the notation $(\Sigma,\nu)$ to denote a pair such that $\Sigma\subset\R^{n+1}$ is immersed and $\nu\in\Gamma(N\Sigma)$ is unit.
\begin{Prop} \label{prop: orthogonal implies isotropic}
    Assume $\Phi_\nu(\Sigma)\in T\S^n$ is immersed and $\dim \Sigma>0$. Then $\Sigma$ is an orthogonal submanifold of $\Phi_\nu(\Sigma)$ and $\Phi_\nu(\Sigma)$ is isotropic.
\end{Prop}
\begin{proof}
The first claim holds by definition. For the second, let $X,Y\in T_p\Sigma$. Then
   \begin{align*}
(\Phi_\nu^\ast\Omega)_p(X,Y)&=\langle \overline\nabla_X \nu(p), Y \rangle-\langle \overline\nabla_Y \nu(p), X \rangle &&\text{(By Lemma \ref{Lem: pullback of omega})}\\
&=\langle [X,Y],\nu(p)\rangle &&\text{(Since $\overline{\nabla}$ is torsion free and $\nu(p)\perp T_p\Sigma$)}\\
    &=0. &&\text{(Since $\nu(p)\perp T_p\Sigma$)}
    \end{align*}
    Since $\nu\in\Gamma(N\Sigma)$, $\Phi_\nu$ is an immersion by Lemma \ref{Lem: derivative of Phi_v} and $\Omega|_{\Phi_\nu(\Sigma)}=0$ follows.
\end{proof}
\begin{Rem}
    Proposition \ref{prop: orthogonal implies isotropic} is a converse to Proposition \ref{prop: isotropic implies existance of orthogonal sub manifolds}. Indeed if $\Gamma\subseteq T\S^n$ admits an orthogonal submanifold $\Sigma\subset\mathbb{R}^{n+1}$ with a unit normal field $\nu$, $\Phi_\nu(\Sigma)$ is isotropic by Proposition \ref{prop: orthogonal implies isotropic}. If $\Gamma=\Phi_\nu(\Sigma)$ - i.e. $\Gamma$ is entirely generated by the normal vector field $\nu\in\Gamma(N\Sigma)$ - then $\Gamma$ itself is isotropic.
\end{Rem}
\begin{Prop} \label{Prop: orthgononal submanifolds distance from focal set is radii of curvature}
 Suppose $\Phi_\nu(\Sigma)\subset T\S^n$ is immersed. Let $q\in \Sigma$ and let $p\in\mathbb{R}^{n+1}\backslash\{q\}$ lie on the line $\Phi_\nu(q)$. Then, $p$ is a focal point of $\Phi_\nu(\Sigma)$ iff there exists a non-zero $Y\in T_q\Sigma$ satisfying the eigenvalue equation
 \[
 \langle q-p,\nu(q)\rangle \overline\nabla_Y \nu(q)=Y.
 \]
 In particular, the eigenvalue is the distance along $\nu(q)$ from $\Sigma$ to the focal set of $\Phi_\nu(\Sigma)$. If $\Sigma$ is a hypersurface, then $X\mapsto \overline{\nabla}_X\nu$ is the shape operator of $\Sigma$.
\end{Prop}
\begin{proof}
Let $q\in \Sigma$ and $p\in\mathbb{R}^{n+1}\backslash\{q\}$ belong to the line $\Phi_\nu(q)$. Hence we may write $p=\Psi(\Phi_\nu(q),\langle p,\nu(q)\rangle)$, i.e. $p=q-\langle q-p,\nu(q)\rangle \nu(q)$.
If $Y\in T_q\Sigma$, then by Lemma \ref{Lem: derivative of Phi_v}
\[
\Phi^\ast_\nu\Big(\langle p,\nu(q)\rangle\td\pi+\mathcal{K}\Big)(Y)=\langle p-q,\nu(p)\rangle\overline\nabla_Y\nu(q)+Y,
\]
and $\Phi_\nu$ is an immersion. The claim now follows from Proposition \ref{prop: characterisation of focal points}.
\end{proof}
\section{Lines Tangent to Submanifolds}
\label{sec: lines tangent submanifolds}
Now we consider lines generated by the pair $(S,v)$ when $v$ is a unit tangent vector field over $S$, and again $S\subset\R^{n+1}$ is immersed. In this case we prove that $S$ is a part of the focal set of $\Phi_v(S)$ and provide some characterisations of $\Phi_v(S)$ being isotropic. First a characterisation of when $\Phi_v$ is an immersion is given.
\begin{Lem}\label{lem: dphi is injective iff v is nowhere geodesic}
     The map $\Phi_v$ is an immersion if and only if $v\in\Gamma(TS)$ is nowhere geodesic with respect to the Euclidean connection $\overline\nabla$ (i.e. $\overline\nabla_{v}v(p)\neq 0$ for all $p\in S$).
\end{Lem}
\begin{proof}
    We prove the contrapositive, namely for each $p\in S$
    \[
    \td\Phi_v|_p \text{ has non trivial kernel } \iff \overline\nabla_{v}v(p)=0.
    \]

    The backwards implication follows by letting $X=v(p)$ in Lemma \ref{Lem: derivative of Phi_v}. To prove the forward implication let $X\in\mbox{Ker }\td \Phi_v|_p$, with $X\neq 0$. Lemma \ref{Lem: derivative of Phi_v} implies $\overline\nabla_Xv(p)=0$ and $X=\lambda v(p)$ with $\lambda\neq 0$. Since $\overline\nabla$ is tensorial we have $\overline\nabla_{v}v(p)=0$.
\end{proof}
Hence $\Phi_v$ fails to be an immersion precisely when the integral curves of $v$, thought of as curves in $\R^{n+1}$, have zero curvature.
\begin{Prop}\label{prop: submanfolds tangent to the lines are focal}
    Suppose $\Phi_v$ is an immersion and $\Phi_v(S)\subset T\S^n$ is immersed. The focal set of $\Phi_v(S)$ contains $S$.
\end{Prop}
\begin{proof}
     Given any $p\in S$, note that $p=\Psi(\Phi_v(p),\langle p,v(p)\rangle)\in \Psi(\Phi_v(S)\times \mathbb{R})$. Let $X=\td\Phi_v(v(p))\in T_{\Phi_v(p)}\Phi_v(S)$, which is non zero by Lemma \ref{lem: dphi is injective iff v is nowhere geodesic}. One quickly shows
     \begin{align*}
     \langle p,v(p)\rangle \td\pi(X)+\mathcal{K}(X)&=0,
     \end{align*}
     by an application of Lemma \ref{Lem: derivative of Phi_v}. Applying Proposition \ref{prop: characterisation of focal points} completes the proof.
\end{proof}
\begin{Mainthm}{1}\label{thm: Characterisations of Isotropic For Tangent Lines}
Suppose $\Phi_v$ is an immersion and $\Phi_v(S)\subset T\S^n$ is immersed. The following are equivalent.
    \begin{enumerate}[label=(\alph*), itemsep=4pt, topsep=4pt]
        \item $\Phi_v(S)$ is isotropic.
        \item Each $p\in S$ belongs to an open set $U\subseteq S$ on which $v|_U=\mbox{grad~}f$ for some smooth function $f:U\to \mathbb{R}$. If $S$ is simply connected we may take $U=S$.
        \item The distribution $v^\bot$ is integrable and $v$ is a geodesic vector field on $S$
    \end{enumerate}
\end{Mainthm}
\begin{proof}
    $(a)\!\implies\!(b)$: Define the $1$-form $ A:p \mapsto \langle ~\cdot~, v(p)\rangle\Big|_{T_pS}~
    $ on $S$. Since $\overline{\nabla}$ is the Levi-Civita connection one quickly shows that for $X,Y\in \Gamma(TS)$,
    \begin{align*}
    \td A(X,Y)&=X(A(Y))-Y(A(X))-A([X,Y])\\
    &=\langle X,\overline{\nabla}_Yv\rangle-\langle Y,\overline{\nabla}_X v\rangle &&\text{(Since $\overline{\nabla}$ is $\mathbb{R}^{n+1}$'s Levi-Civita connection.)}\\
    &=\langle Y,\nabla_Xv\rangle-\langle X,\nabla_Y v\rangle &&\text{(Since $\nabla$ is $\overline{\nabla}$ restricted to $S$)}\\
    &=-\Phi^\ast\Omega(X,Y). &&\text{(By Lemma \ref{Lem: pullback of omega} )}
    \end{align*}
    Thus, since $\Phi$ is an immersion
    \[
    \Omega|_{\Phi_v(S)}=0 \implies \Phi^\ast\Omega=0 \implies \td A=0.
    \]
    The Poincaré Lemma now implies for each $p\in S$ there is an open neighbourhood $U$ and a smooth function $f:U\to \mathbb{R}$ such that with $A|_U=\td f$, i.e. $v|_U=\mbox{grad~}f$.\\
    $(b) \! \implies \! (c)$: First we show $v$ is a geodesic vector field on $S$. Let $p\in S$ and $v|_U=\mbox{grad~}f$ with $U$ and $f:U\to\R$ as above. Then for any $X\in T_pS$,
    \[
    \langle \nabla_{v}v(p),X\rangle=\mbox{Hess}(f)_p(v(p),X)=\mbox{Hess}(f)_p(X,v(p))=\langle \nabla_Xv(p),v(p)\rangle=0,
    \]
    as $v$ is unit. Since $p$ and $X$ were arbitrary, $\nabla_vv=0$. To show
    \[
        v^\bot:p\mapsto \{X\in T_pS~:~\langle X,v(p) \rangle=0\},
    \]
        is integrable, let $c:=f(p)$. Since $||\mbox{grad~}f||=||v||=1$, it follows that $\mbox{rank~}\td f=1$ and $c$ is a regular value of $f$. Hence $f^{-1}(c)$ is an embedded submanifold of $U$ of dimension $\dim S-1$ containing $p$. It is then easily checked that $T_pf^{-1}(c)=v^\bot(p)$.\\
    $(c)\! \implies \! (a)$: Due to the properties of $v$, both $v^\perp$ and $\mbox{Span~}v$ are regular, smooth subbundles of $TS$ and we have the bundle splitting
    \[
    TS=v^\perp\oplus \mbox{Span~}v.
    \]
    Now take any $X,Y\in \Gamma(TS)$ and decompose them via the above:
    \begin{align*}
        &X=\widetilde{X}+\alpha v &&\text{and} &&Y=\widetilde{Y}+\beta v,
    \end{align*}
    with $\widetilde{X},\widetilde{Y}\in \Gamma( v^\perp)$ and $\alpha,\beta$ smooth functions on $S$. It follows that
    \begin{align*}
        \Phi^\ast\Omega(X,Y)=\Phi^\ast\Omega(\widetilde{X},\widetilde{Y})+\alpha\Phi^\ast\Omega(v,\widetilde{Y})+\beta\Phi^\ast\Omega(\widetilde{X},v).
    \end{align*}
     Applying Lemma \ref{Lem: pullback of omega} and the conditions that $v$ is geodesic and unit gives
    \[
    \Phi^\ast\Omega(v,\widetilde{Y})=\langle \nabla_vv,\widetilde{Y}\rangle+\langle v,\nabla_{\widetilde{Y}}v\rangle=0,
    \]
    and similarly $\Phi^\ast\Omega(\widetilde{X},v)=0$. Also since $v^\perp$ is involutive, $[\widetilde{X},\widetilde{Y}]\in v^\perp$, hence
    \begin{align*}
        \Phi^\ast\Omega(\widetilde{X},\widetilde{Y})&=\langle \nabla_{\widetilde{X}}v,\widetilde{Y}\rangle-\langle \nabla_{\widetilde{Y}}v,\widetilde{X}\rangle &&\text{(Lemma \ref{Lem: pullback of omega})}\\
        &=\widetilde{X}(\langle v,\widetilde{Y}\rangle)-\widetilde{Y}(\langle v,\widetilde{X}\rangle)-\langle v,[\widetilde{X},\widetilde{Y}]\rangle &&\text{(Since $\nabla$ is Levi-Civita)}\\
        &=0. &&\text{(Since $\widetilde{X},\widetilde{Y}, [\widetilde{X},\widetilde{Y}]\in v^\perp$)}
    \end{align*}
     It follows that $\Phi^\ast\Omega(X,Y)=0$. Since $X$ and $Y$ were arbitrary, the claim follows.
\end{proof}
\begin{Rem}\label{rem: local dist func}
    The function $f:U\to \mathbb{R}$ in Theorem \ref{thm: Characterisations of Isotropic For Tangent Lines} satisfies $||\mbox{grad~}f||=1$. Such functions are called \textit{local distance functions}. Indeed if $c\in\R$ is such that $f^{-1}(c)\neq \emptyset$, then for all $x\in U$ sufficiently close to $f^{-1}(c)$, $|f(x)-c|$ is the geodesic distance (in $U$) from $x$ to $f^{-1}(c)$ (see \cite{Lee18}, pp.178).
\end{Rem}

\begin{Cor}
    Suppose $S\subset \mathbb{R}^{n+1}$ is a surface (i.e. a two dimensional immersed submanifold) and the pair $(S,v)$ satisfy the conditions of Theorem \ref{thm: Characterisations of Isotropic For Tangent Lines}. $\Phi_v(S)$ is isotropic if and only if $v$ is a geodesic vector field on $S$. If $n=2$, ``isotropic'' may be replaced with ``Lagrangian''.
\end{Cor}
\begin{proof}
By assumption, $v^\perp$ has dimension $\dim S-1=1$ and therefore is integrable.  Now apply $(a)\iff (c)$ of the above theorem. If $n=2$, since $\Phi_v(S)$ is a $2$-dimensional submanifold of $TS^2$, it is isotropic iff it is Lagrangian.
\end{proof}
\section{The Curvatures of Focal Sheets}\label{sec: The Curvature of Focal Sets}
We now explore the geometry of focal sets when they possess a $C^2$-smooth manifold structure and their associated family of lines is isotropic. To this end, we consider two immersed submanifolds of $\R^{n+1}$ equipped with unit tangent vector fields, namely the pairs $(S_i,v_i)$ for $i=1,2$, which satisfy the conditions
\begin{enumerate}
    \item \label{cond: Phi1=Phi2} $\Phi_{v_1}(S_1)=\Phi_{v_2}(S_2)$,
    \item \label{cond: Phi is diffeo} $\Phi_{v_i}:S_i\to \Phi_{v_i}(S_i)$ is a diffeomorphism,
     \item \label{cond: Phi isotropic} $\Phi_{v_i}(S_i)$ is isotropic.
\end{enumerate}
Condition (\ref{cond: Phi1=Phi2}) ensures that both pairs $(S_i,v_i)$ generate the same family of oriented lines, namely $\Phi_{v_i}(S_i)$, for either $i=1$ or $2$. Since the vector field $v_i$ is tangent to $S_i$, Proposition \ref{prop: submanfolds tangent to the lines are focal} shows both $S_1$ and $S_2$ are contained in the focal set of $\Phi_{v_i}(S_i)$. Condition (\ref{cond: Phi is diffeo}) ensures the submanifolds $S_i$ can be related to one another through their common image under $\Phi_{v_i}$. Indeed $S_1$ is diffeomorphic to $S_2$ under the map $\Phi_{v_2}^{-1}\circ \Phi_{v_1}:S_1\to S_2$, and if $p_i\in S_i$ are related by $\Phi_{v_1}(p_1)=\Phi_{v_2}(p_2)$ then
\begin{align}\label{equ: equate images of Phi's}
    &v_1(p_1)=v_2(p_2), &&\text{and} &&p_1-\langle p_1,v_1(p_1)\rangle=p_2-\langle p_2,v_2(p_2)\rangle,
\end{align}
by definition of $\Phi_{v_i}$. Therefore we may write
\[
p_1-p_2=\langle p_1-p_2,v_1(p_1)\rangle v_1(p_1),
\]
with $\langle p_1-p_2,v_1(p_1)\rangle$ being the signed distance in $\mathbb{R}^{n+1}$ between $S_1$ and $S_2$ along the line $v_1(p_1)$. Motivated by this we define the signed distance function
\begin{equation}\label{equ: distance function}
\begin{aligned}
\mu_{12}:~ &S_1\to \mathbb{R}\\
&p\longmapsto  \left\langle p- (\Phi_{v_2}^{-1} \circ \Phi_{v_1})(p),v_1(p)\right\rangle.
\end{aligned}
\end{equation}
from $S_1$ to $S_2$. It follows that $(\Phi_{v_2}^{-1}\circ\Phi_{v_1})(p)=p-\mu_{12} (p)v(p)$, so $\mu_{12}(p)\neq 0$ iff $p$ and $(\Phi_{v_2}^{-1}\circ\Phi_{v_1})(p)$ are distinct focal points on the line $\Phi_{v_1}(p)$.
\begin{Prop}\label{prop: E_2 is eigen value of shape op on S1}
    Let $\nabla$ be the connection on $S_1$ induced by $\overline{\nabla}$ and define
    \begin{align*}
        &E_2:=V_2-\langle V_2,v_1\rangle v_1 &&\text{where} &&V_2:=\td(\Phi^{-1}_{v_1}\circ\Phi_{v_2})(v_2).
    \end{align*}
    Suppose $\mu_{12}$ is nowhere $0$. Then $E_2$ is non-zero and satisfies 
    \begin{equation}\label{equ: nabla_E_2 v_1}
    \nabla_{E_2}v_1=\mu_{12}^{-1}E_2.
\end{equation}
Furthermore, if $p\in \mathcal{M}$ with $\mathcal{M}$ an integrable manifold of $v_1^\perp$, then $E_2(p)\in T_p\mathcal{M}$ is a principal direction of $\mathcal{M}$ with principal curvature $1/\mu_{12}(p)$.
\end{Prop}
\begin{proof}
    Let $p_i\in S_i$ satisfy $\Phi_{v_1}(p_1)=\Phi_{v_2}(p_2)$. Also $\td\Phi_{v_1}|_{p_1}(V_2)=\td\Phi_{v_2}|_{p_2}(v_2)$ holds by definition of $V_2$. After projecting to $T\S^n$ via $\td\pi$ and $\mathcal{K}$, Lemma \ref{Lem: derivative of Phi_v} gives\\
\begin{equation}
\hspace{-1.5em}
\label{equ: equate images of d_Phi's}
\left \{\begin{aligned}
&\overline\nabla_{V_2}v_1(p_1)=\overline\nabla_{v_2}v_2(p_2)\\
&V_2(p_1\!)\!-\!\langle V_2(p_1\!),\!v_1\!(p_1\!)\rangle v_1\!(p_1\!)\!-\!\langle p_1,\!v_1\!(p_1\!)\rangle \overline{\nabla}_{V_2}v_1\!(p_1\!)\!=\!-\langle p_2,\!v_2(p_2\!)\rangle \overline{\nabla}_{v_2}v_2(p_2\!).
\end{aligned}\right.
\end{equation}
Furthermore by combining the these equations with equations (\ref{equ: equate images of Phi's}), one derives that
\[V_2(p_1)=\langle V_2(p_1),v_1(p_1)\rangle v_1(p_1)+\langle p_1-p_2, v_1(p_1) \rangle\overline{\nabla}_{V_2}v_1(p_1),
\]
which implies that $\overline{\nabla}_{V_2}v_1(p_1)\in T_{p_1}S_1$ and thus $\overline{\nabla}_{V_2}v_1(p_1)=\nabla_{V_2}v_1(p_1)$. The above decomposes $V_2$ into its components orthogonal and parallel to $v_1$;
\begin{equation}\label{eq: V2 in decomposition}
    V_2=\langle V_2, v_1\rangle v_1+\mu_{12} \nabla_{V_2}v_1.
\end{equation}
    Then,
    \begin{align*}
        E_2&=V_2-\langle V_2, v_1\rangle v_1 &&\text{(By definition of $E_2$)}\\
        &=\mu_{12} \nabla_{V_2}v_1 &&\text{(Identity (\ref{eq: V2 in decomposition}))}\\
        &=\mu_{12}(\langle V_2, v_1\rangle \nabla_{v_1}v_1+\nabla_{E_2}v_1) &&\text{(By definition of $E_2$)}\\
        &=\mu_{12} \nabla_{ E_2} v_1. &&\text{(Condition (\ref{cond: Phi isotropic}) and Main Theorem \ref{thm: Characterisations of Isotropic For Tangent Lines})} 
    \end{align*}
As $\mu_{12}\neq 0$, equation (\ref{equ: nabla_E_2 v_1}) follows. Also, by the above calculation and equations (\ref{equ: equate images of d_Phi's})
\[
E_2=\mu_{12}\nabla_{V_2}v_1=\mu_{12}\overline{\nabla}_{v_2}v_2.
\]
Since $\Phi_{v_2}$ is a diffeomorphism, $\overline\nabla_{v_2}v_2\neq 0$ by Lemma \ref{lem: dphi is injective iff v is nowhere geodesic}, thus $E_2\neq 0$. Finally, if $\mathcal{M}$ is an integral manifold of $v_1^\bot$ (which exists by Main Theorem \ref{thm: Characterisations of Isotropic For Tangent Lines}), then $\mathcal{M}$ is a hypersurface in $S_1$ with unit normal $v_1|_\mathcal{M}$ and shape operator $X\mapsto \nabla_X v_1$. If $p\in \mathcal{M}$, then $E_2(p)\in T_p\mathcal{M}$ as $E_2\perp v_1$.  Evaluating equation (\ref{equ: nabla_E_2 v_1}) at $p$ completes the proof.
\end{proof}
\begin{Thm} \label{Thm: sectional curvature from mu}
    Suppose $\mu_{12}$ is nowhere $0$. The sectional curvature $K$ of $S_1$ satisfies
    \[
    K(V_2,v_1)=\frac{\td\mu_{12}(v_1)-1}{\mu_{12}^2}.
    \]
\end{Thm}

\begin{proof}
         Note that since $v_1\wedge V_2$ and $v_1 \wedge E_2$ define the same plane, we have
         \[
         K(V_2,v_1)=K(E_2,v_1).
         \]
         Furthermore, after rescaling if necessary we may assume $E_2$ is unit - in which case equation (\ref{equ: nabla_E_2 v_1}) continues to hold. It follows that $\{E_2,v_1\}$ are orthonormal and
         \begin{align*}
         K(E_2,v_1)&=\langle R(E_2,v_1)v_1,E_2\rangle &&\text{(Where $R(\cdot,\cdot,\cdot))$ is the Riemann tensor)}\\
         &=-\Big(\langle \nabla_{v_1}\nabla_{E_2}v_1,E_2\rangle+\langle\nabla_{[E_2,v_1]}v_1,E_2\rangle\Big). && \text{(Since $\nabla_{v_1}v_1=0$ by Main Theorem \ref{thm: Characterisations of Isotropic For Tangent Lines})}
         \end{align*}
            Calculating these two terms gives
         \begin{align*}
             \langle \nabla_{v_1}\nabla_{E_2}v_1,E_2\rangle&=\langle \nabla_{v_1}(\mu_{12}^{-1}E_2),E_2\rangle &&\text{(By Equation \ref{equ: nabla_E_2 v_1})}\\
             &=\langle 
             -\mu_{12}^{-2}\td \mu_{12}(v_1)E_2+\mu_{12}^{-1}\nabla_{v_1}E_2,E_2\rangle\\
             &=-\mu_{12}^{-2}\td \mu_{12}(v_1)+\mu_{12}^{-1}\langle\nabla_{v_1}E_2,E_2\rangle.\\ \\
             \langle\nabla_{[E_2,v_1]}v_1,E_2\rangle&=\langle \nabla_{E_2}v_1,[E_2,v_1]\rangle &&\text{(By Lemma \ref{Lem: pullback of omega}, since $\Omega|_{\Phi_{v_1}(S_1)}=0$)}\\
             &=\langle \nabla_{E_2}v_1,\nabla_{E_2}v_1\rangle
-\langle \nabla_{E_2}v_1,\nabla_{v_1}E_2\rangle &&\text{($\nabla$ is torsion free)}\\
&=\mu_{12}^{-2}-\mu_{12}^{-1}\langle E_2, \nabla _{v_1}E_2\rangle. &&\text{(By Equation (\ref{equ: nabla_E_2 v_1}))}
         \end{align*}
         Insertion of the above into the formula for $K(E_2,v_1)$ yields
         \begin{align*}
         K(E_2,v_1)
=\frac{\td \mu_{12}(v_1)-1}{\mu_{12}^2}.
         \end{align*}
    \end{proof}
If more of the focal set is known, further control of the geometry of $S_1$ is uncovered.
\begin{Thm} \label{Thm: curvature control from multiple focal set componants}
   Let $(S_1,v_1),\ldots,(S_k,v_k)$ satisfy pairwise conditions $(\ref{cond: Phi1=Phi2})$ - $(\ref{cond: Phi isotropic})$, with $\dim S_i=k$. Define $\mu_{ij}:S_i\to \mathbb{R}$, the signed distance from $S_i$ to $S_j$ by
\begin{equation}\label{eq: all the distance functions}
\mu_{ij}(q)=\left\langle q- (\Phi_{v_j}^{-1} \circ \Phi_{v_i})(q),v_i(q)\right\rangle, \qquad i,j\in\{1,\ldots,k\}.
\end{equation}
Assume $\mu_{ij}$ is nowhere $0$ for $i\neq j$ and define $E_{ij},V_{ij}\in\Gamma(TS_i)$ by
\begin{align*}
        &E_{ij}:=V_{ij}-\langle V_{ij},v_i\rangle v_i &&\text{where} &&V_{ij}:=\td(\Phi^{-1}_{v_i}\circ\Phi_{v_j})(v_j).
    \end{align*}
\begin{enumerate}[label=\alph*)]
    \item Let $\mathcal{M}\subset S_1$ be an integrable manifold of $v_1^\perp$ and $p\in \mathcal{M}$. Then $\{E_{1j}(p)\}^k_{j=2}$ are principal directions of $\mathcal{M}$ at $p$ with distinct principal values $\{1/\mu_{1j}(p)\}_{j=2}^{k}$.
    \item The set of vector fields $\{E_{1j}\}^k_{j=2}$ are an orthogonal frame of $v_1^\perp$ and the Ricci curvature tensor of $S_1$ satisfies
    \[
    \mbox{Ric}(v_1,v_1)=\sum_{j=2}^{k}K(v_1,E_{1j})=\sum_{j=2}^{k}\frac{\td \mu_{1j}(v_1)-1}{(\mu_{1j})^2}.
    \] 
\end{enumerate}
\end{Thm}
\begin{proof}
(a): Let $q\in S_1$, first we show that $\mu_{1j}(q)\neq \mu_{1i}(q)$ when $i\neq j$ Define $q_j=\Phi_{v_j}^{-1}\circ\Phi_{v_1}(q)\in S_{j}$ and $q_i=\Phi_{v_i}^{-1}\circ\Phi_{v_1}(q)\in S_i$ so that $\Phi_{v_i}(q_i)=\Phi_{v_j}(q_j)=\Phi_{v}(q)$. Then $v_1(q)=v_j(q_j)=v_i(q_i)$ and from the definitions of $\mu_{1j},\mu_{1i}$ and $\mu_{ij}$ one finds
\[
\mu_{1j}(q)-\mu_{1i}(q)=\mu_{ij}(q_i).
\]
Hence $\mu_{1j}(q)\neq\mu_{1i}(q)$, otherwise $\mu_{ij}(q_i)=0$, contradicting $0\notin\mu_{ij}(S_i)$. Let $\mathcal{M}$ and $p\in \mathcal{M}$ be as above. By definition $E_{1j}(p)\in v_1^\perp(p)=T_p\mathcal{M}$. Applying Proposition \ref{prop: E_2 is eigen value of shape op on S1} to the pair $(S_1,v_1),(S_j,v_j)$ in place of  $(S_1,v_1),(S_2,v_2)$ proves (a).\\ \hfill

(b) The orthogonality follows since $\{E_{1j}\}_{j=2}^k$ are (pointwise) the principal directions of $\mathcal{M}$ with distinct principal values. Since there are $k-1$ of them, they span $v^\perp$. The Ricci curvature expression follows directly from Theorem \ref{Thm: sectional curvature from mu}.
\end{proof}
\section{The Hypersurface case}\label{sec: the focal sets of hypersurfaces}
Given a hypersurface $\Sigma$, the results from the previous section are now applied to investigate the curvature of its focal sheets. A classical result in this direction is
\begin{Manualthm}{A}[Bianchi, 1874] \label{Bianchi thm}\cite{B1874}
Let $\Sigma\subset\mathbb{R}^3$ be a $C^3$-smooth, oriented, ridge free and strictly convex surface. Suppose the radii of curvature of $\Sigma$ satisfy $r_2-r_1 \equiv s$ for $s\neq0$ a constant. Then the two focal sheets of $\Sigma$ have Gaussian curvature $-1/s^2$.
\end{Manualthm}
Some discussion of the above terminology is warranted. In general focal sheets are singular, see \cite{BT23} for recent work on this. At points where the radii of curvature $r_i:\Sigma\to\R$ have multiplicity $\geq1$, the functions $r_i$ are not in general differentiable and the focal sheets are singular.
\begin{Def}
    Given a hypersurface $\Sigma$ the functions
\begin{align*}
        s_{ij}:~&\Sigma \to \R\\
        &p \to r_i(p)-r_j(p),
\end{align*}
with $i,j\in\{1,\ldots,n\}$ are called the \textit{astigmatisms} of $\Sigma$.
\end{Def}
The multiplicities of all radii of curvature at a point $p\in \Sigma$ are $1$ iff $s_{ij}(p)\neq 0$ for all $i,j\in\{1,\ldots,n\}$, $i\neq j$. Hence the requirement $s\neq0$ in Bianchi's theorem.
\begin{Def} Let $\Sigma$ be a hypersurface with a unit normal field $\nu$ and $r_i:\Sigma\to \R$ a radii of curvature of $\Sigma$. A point $q\in\Sigma$ is called a \textit{ridge point associated to} $r_i$ \cite{Port94} if
    \begin{align}\label{eq: ridge condition}
        &r_i(q)\overline{\nabla}_X \nu=X &&\text{and} &&\td r_i(X)=0,
    \end{align}
    for some non-zero $X\in T_q\Sigma$. $\Sigma$ is called \textit{ridge free} if it has no ridge points.
\end{Def}
The ridge free condition in Bianchi's thoerem ensures the maps $\mathcal{Y}_i$ are immersions, or in other words, that the surface is locally diffeomorphic to its focal sheets.
\begin{Lem}\label{lem: dY_i is injective iff no ridge}
    Let $\Sigma\subset\R^{n+1}$ be a $C^3$-smooth hypersurface. Let $\nu$ be a unit normal vector field and $r_i:\Sigma\to\R$ a radii of curvature with multiplicity $1$. If $X\in T_q\Sigma$ then
    \begin{equation}
        \td\mathcal{Y}_i(X)=X-\td r_i(X)\nu(q)-r_i(q)\overline{\nabla}_X\nu(q).
    \end{equation}
    In particular, $\mathcal{Y}_i$ is an immersion at $q$ iff $q$ is not a ridge point.
\end{Lem}
\begin{proof}
    Since $\Sigma$ is $C^3$-smooth and has no umbilic points, the radii of curvature function  $r_i:\Sigma\to \R$ is differentiable. The expression for $\td\mathcal{Y}_i$ now easily follows from equation (\ref{eq: maps from surface to focal set}). It is also easy to show, due to the orthogonality of $\nu(q)$ with $X$ and $\overline{\nabla}_X\nu(q)$, that $\td\mathcal{Y}_i|_q$ has non-zero kernel iff $q$ is a ridge point.
    \end{proof}
Even with the ridge free condition, a focal sheet may self intersect. Therefore injectivity conditions on the maps $\mathcal{Y}_i$ are natural to ensure a manifold structure. Since the curvature properties of focal sets are discussed, both conditions are assumed - i.e. that the maps $\mathcal{Y}_i$ are diffeomorphisms onto their images.
\begin{Prop}\label{prop: orthogonal hypersurface gives focal sets hypersurface folliation}
    Let $\Sigma\subset \R^{n+1}$ be a $C^3$-smooth, strictly convex hypersurface with unit normal $\nu$ whose radii of curvature $(r_i)_{i=1}^n$ have multiplicity $1$. Let $S_i:=\mathcal{Y}_i(\Sigma)$ be a focal sheet with $\mathcal{Y}_i$ a diffeomorphism. Define $X_i\in\Gamma(T\Sigma)$ by
    \begin{align*}
       &r_i\overline{\nabla}_{X_i}\nu =X_i &&\text{and} &&\td r_i(X_i)=1,
    \end{align*}
    so $X_i$ is a principal vector field. Then $v_i:=-\td\mathcal{Y}_i(X_i)$ is a unit and geodesic vector field over $S_i$ and $v_i^\perp$ is integrable. Also $\Phi_{v_i}\circ\mathcal{Y}_i=\Phi_\nu$, $v_i\circ\mathcal{Y}_i=\nu$ and $\Phi_{v_i}$ is an immersion.
\end{Prop}
\begin{proof}
    We first note $X_i$ is well defined since the radii of curvature have multiplicity $1$ and $\Sigma$ is ridge free. Also $v_i$ is well defined.
    The above claim is simply part $(c)$ of Theorem \ref{thm: Characterisations of Isotropic For Tangent Lines} applied to $(S_i,v_i)$. We thus only need to show that part $(a)$ holds (i.e. $\Phi_{v_i}(S)$ is isotropic), and that $v_i$ satisfies Theorem \ref{thm: Characterisations of Isotropic For Tangent Lines}'s prerequisites, that is, $v_i$ is a unit vector field and $\Phi_{v_i}$ is an immersion. One checks using Lemma \ref{lem: dY_i is injective iff no ridge} that if $q\in\Sigma$,
    \[
    (v_i\circ\mathcal{Y}_i)(q)=-\td\mathcal{Y}_i{\big|_q}(X_i)=-X_i(q)+\td r_i|_q(X_i)\nu(q)+r_i(q)\overline{\nabla}_{X_i}\nu(q)=\nu(q),
    \]
    by definition of $X_i$. Hence $v_i\circ\mathcal{Y}_i=\nu$ and so $v_i$ is unit. If $p\in S_i$ and $q\in \Sigma$ satisfy $p=\mathcal{Y}_i(q)$, then
    \begin{align*}
    \Phi_{v_i}(p)&=(v_i(p),p-\langle p,v_i(p)\rangle v_i(p)), &&\text{(By definition of $\Phi_{v_i}$)}\\
    &=(\nu(q), \mathcal{Y}_i(q)-\langle\mathcal{Y}_i(q),\nu(q)\rangle \nu(q)), &&\text{(By definition of $p$ and $v_i\circ \mathcal{Y}_i=\nu$)}\\
    &=(\nu(q),q-\langle q, \nu(q) \rangle), &&\text{(By definition of $\mathcal{Y}_i$)}\\
    &=\Phi_\nu(q).
    \end{align*}
    Hence $\Phi_{v_i}\circ \mathcal{Y}_i=\Phi_\nu$ and so $\Phi_{v_i}(S_i)=\Phi_\nu(\Sigma)$. However, $\Sigma$ is an orthogonal submanifold of $\Phi_\nu(\Sigma)$, thus $\Phi_\nu(\Sigma)$, and hence $\Phi_{v_i}(S_i)$, is isotropic. Also by the above $\Phi_{v_i}=\Phi_\nu\circ\mathcal{Y}_i^{-1}$. Thus $\Phi_{v_i}$ is an immersion as $\Phi_\nu$ is (by Lemma \ref{Lem: derivative of Phi_v}).
\end{proof}
Together, Theorem \ref{Thm: sectional curvature from mu} and Proposition \ref{prop: orthogonal hypersurface gives focal sets hypersurface folliation} allow us to relate together the geometry of different focal sheets of a fixed hypersurface as follows.
\begin{Mainthm}{2}\label{thm: sectional curvature of hypersurface focal set}
Let $\Sigma\subset \R^{n+1}$ be a $C^3$-smooth, strictly convex hypersurface with unit normal field $\nu$ whose radii of curvature $(r_i)_{i=1}^n$ have multiplicity $1$. For $i=1,2$ let $S_i:=\mathcal{Y}_i(\Sigma)$ be focal sheets with $\mathcal{Y}_i$ a diffeomorphism. If $X_i\in\Gamma(T\Sigma)$ satisfy
    \begin{align*}
       &r_i\overline{\nabla}_{X_i}\nu =X_i &&\text{and} &&\td r_i(X_i)=1,
    \end{align*}
      The sectional curvature of $S_1$, denoted $K_1$, satisfies
    \begin{equation}\label{equ: sectional curvature of focal set}
    (\mathcal{Y}_1^\ast K_1)\big(X_1,X_2)=\frac{\td s_{12}(X_1)-1}{(s_{12})^2}.
    \end{equation}
\end{Mainthm}
\begin{proof}
    Applying Proposition \ref{prop: orthogonal hypersurface gives focal sets hypersurface folliation} to each focal set $S_i:=\mathcal{Y}(\Sigma)$ for $i=1,2$ gives a pair $(S_i,v_i)$ where $v_i\in\Gamma(TS_i)$ is a unit vector field defined by $v_i:=-\td\mathcal{Y}_i(X_i)$ and satisfying $\Phi_{v_i}\circ \mathcal{Y}_i=\Phi_\nu$, $v_i\circ\mathcal{Y}_i=\nu$ and with each $\Phi_{v_i}$ an immersion. We first show that this pair satisfy the conditions introduced at the start of Section \ref{sec: The Curvature of Focal Sets};
    \begin{enumerate}
    \item $\Phi_{v_1}(S_1)=\Phi_{v_2}(S_2)$.
    \item $\Phi_{v_i}:S_i\to \Phi_{v_i}(S_i)$ is a diffeomorphism.
     \item $\Phi_{v_i}(S_i)$ is isotropic.
\end{enumerate}
Begin by noting that since $\Phi_{v_i}\circ \mathcal{Y}_i=\Phi_\nu$ we have $\Phi_{v_1}\circ \mathcal{Y}_1=\Phi_{v_2}\circ\mathcal{Y}_2$. Hence $\Phi_{v_1}(S_1)=\Phi_{v_2}(S_2)$, which is (\ref{cond: Phi1=Phi2}). For (\ref{cond: Phi is diffeo}) we remark that as curvature is a local notion, to derive equation (\ref{equ: sectional curvature of focal set}) it suffices to derive
\[
K_1\big(\td\mathcal{Y}_1(X_1),\td\mathcal{Y}_1(X_2))=\frac{\td s_{12}(X_1)-1}{(s_{12})^2},
\]
on some open neighbourhood about any point of $S_1$. Hence by working locally we may assume that $\Phi_{v_1}$ is a diffeomorphism once its domain is suitably restricted. Since $\Phi_{v_2}=\Phi_{v_1}\circ \mathcal{Y}_1\circ\mathcal{Y}_2^{-1}$, the same domain can be pulled back under $\mathcal{Y}_1\circ\mathcal{Y}_2^{-1}$ on which $\Phi_{v_2}$ is a diffeomorphism, thus (\ref{cond: Phi is diffeo}). Condition (\ref{cond: Phi isotropic}) follows from (\ref{cond: Phi1=Phi2}) since $\Phi_\nu(\Sigma)$ is isotropic. To calculate the signed distance between $S_1$ and $S_2$ (Cf. equation (\ref{equ: distance function})) let $p\in S_1$ and $q\in\Sigma$ be related by $p=\mathcal{Y}_1(q)$. Then
\begin{align*}
\mu_{12}(p)&:=\langle p-(\Phi_{v_2}^{-1}\circ \Phi_{v_1})(p),v_1(p)\rangle,\\
&=\langle \mathcal{Y}_1(q)-(\Phi_{v_2}^{-1}\circ \Phi_{v_1}\circ\mathcal{Y}_1)(q),\nu(q)\rangle, &&\text{(Using $v_1\circ\mathcal{Y}_1=\nu$)}\\
&=\langle \mathcal{Y}_1(q)-\mathcal{Y}_2(q),\nu(q)\rangle, &&\text{(Using $\Phi_{v_2}^{-1}\circ\Phi_{v_1}=\mathcal{Y}_2\circ\mathcal{Y}_1^{-1}$)}\\
&=r_1(q)-r_2(q) &&\text{(Using the definition of $\mathcal{Y}_i$)}\\
&=-s_{12}(q).
\end{align*}
Hence $\mu_{12}\circ\mathcal{Y}_1=-s_{12}$. Since $\Sigma$ has radii of curvatures which have multiplicity $1$, $s_{12}$, and thus $\mu_{12}$, are nowhere $0$. Hence we apply Theorem \ref{Thm: sectional curvature from mu} to find
\[
    K(V_2,v_1)=\frac{\td\mu_{12}(v_1)-1}{\mu^2},
\]
where $V_2=\td(\Phi^{-1}_{v_1}\circ\Phi_{v_2})(v_2)$. However, 
\[
v_2=\td(\Phi^{-1}_{v_1}\circ\Phi_{v_2})(v_2)=\td(\mathcal{Y}_1\circ\mathcal{Y}^{-1}_2)(v_2)=-\td \mathcal{Y}_1(X_2),
\]
by using the definition of $v_2$ and the derived relationships between $\Phi_{v_i}$ and $\mathcal{Y}_i$. Using the above together with $v_1=-\td\mathcal{Y}_1(X_1)$ and $\mu_{12}\circ\mathcal{Y}_1=-s_{12}$, lets one write the above sectional curvature formula as equation (\ref{equ: sectional curvature of focal set}).
\end{proof}
\begin{Mainthm}{3}\label{Thm: focal set of hypersurf}
    Let $\Sigma\subset \R^{n+1}$ be a $C^3$-smooth, strictly convex hypersurface with unit normal field $\nu$ whose radii of curvature $(r_i)_{i=1}^n$ have multiplicity $1$. Let $S_j:=\mathcal{Y}_j(\Sigma)$ be focal sheets with $\mathcal{Y}_j$ a diffeomorphism for all $j\in\{1,\ldots,n\}$. If $X_j\in\Gamma(T\Sigma)$ satisfy
    \begin{align*}
       &r_j\overline{\nabla}_{X_j}\nu =X_j &&\text{and} &&\td r_j(X_j)=1,
    \end{align*}
    the Ricci curvature of $S_i$ ($i\in\{1,\ldots,n\}$) satisfies
    \[
    (\mathcal{Y}_i^\ast\mathrm{Ric})(X_i,X_i)=\sum_{\substack{j=1\\j\neq i}}^n\frac{X_i(s_{ij})-1}{(s_{ij})^2},
    \]
\end{Mainthm}
\begin{proof}
    Fix $i$ as in the theorem statement. We proceed along similar lines as those in Theorem \ref{thm: sectional curvature of hypersurface focal set}, now applying Theorem \ref{Thm: curvature control from multiple focal set componants}, instead of Theorem \ref{Thm: sectional curvature from mu}. For each $j\in \{1,\ldots,n\}$, set $v_j:=-\td\mathcal{Y}_j(X_j)$. Then $(S_1,v_1),\ldots,(S_n,v_n)$ satisfy the prerequisites of Theorem \ref{Thm: curvature control from multiple focal set componants}, that is $v_j\in\Gamma(TS_j)$ are unit and conditions (\ref{cond: Phi1=Phi2})-(\ref{cond: Phi isotropic}) are satisfied. Furthermore $\mu_{ij}\circ\mathcal{Y}_i=-s_{ij}$ and so $\mu_{ij}$ is nowhere $0$. Hence by Theorem \ref{Thm: curvature control from multiple focal set componants} (with $(S_i,v_i)$ in place of $(S_1,v_1)$) gives
    \[
    \mbox{Ric}(v_i,v_i)=\sum_{\substack{j=1\\j\neq i}}^{n}\frac{\td \mu_{ij}(v_i)-1}{(\mu_{ij})^2}.
    \]
    Inserting $v_i=-\td \mathcal{Y}_i(X_i)$ and $\mu_{ij}\circ\mathcal{Y}_i=-s_{ij}$ in the above completes the proof.
\end{proof}
\begin{Rem}\label{rem: Thm 2 with arb codim}
    We remark that a similar results holds when $\Sigma$ has codimension greater than $1$, with minor tweaking of the above proofs. In particular there is no preferred normal unit vector field $\nu\in\Gamma(N\Sigma)$ and the vector fields $X_i\in\Gamma(T\Sigma)$ and functions $r_i:\Sigma\to \R$ are now just solutions of the following eigenproblem on $\Sigma$
    \begin{align*}
        &r_i\overline{\nabla}_{X_i} \nu=X_i &&\text{and} &&X_i(r_i)=1.
    \end{align*}
    We note that solutions exist, if $\Sigma$ is assumed to satisfy the ridge free condition (which is trivially generalised to higher codimension) and $s_{ij}:=r_i-r_j$ is assumed non vanishing, so that the multiplicities of the above eigenvalue problem are $1$.
\end{Rem}


\begin{thebibliography}{CS79}


\bibitem{AGK11}
D. V. Alekseevsky, B. Guilfoyle, and W. Klingenberg, {\it On the geometry of spaces of oriented geodesics}, Ann. Global Anal. Geom. {\bf 40} (2011) 389--409. 

\bibitem{AGK11a}
D.V. Alekseevsky, B. Guilfoyle and W. Klingenberg, {\it Erratum to: On the geometry of spaces of oriented geodesics}, Ann. Global Anal. Geom. {\bf 50.1} (2016) 97--99.  

\bibitem{Aminov1978}
Y.A. Aminov, {\it The Bianchi transformation for a domain of a many-dimensional Lobachevsky space}, Ukr. Geom. Sb. \textbf{21} (1978) 3--5.

\bibitem{Aminov2011}
Y.A. Aminov, {\it On Submanifolds with Negative Curvature in Euclidean Space}, Results Math. \textbf{60} (2011) 117–-131.

\bibitem{AB18}
H. Anciaux and P. Bayard, \textit{On the affine Gauss maps of submanifolds of Euclidean space}, Tsukuba J. Math. \textbf{42} (2018) 397–-415.

\bibitem{A90}
V. I. Arnol’d, \textit{Singularities of caustics and wave fronts}, Mathematics and its Applications (Soviet Series),vol. 62,Kluwer Academic Publishers Group, Dordrecht, 1990.

\bibitem{Arnold94}
V.I. Arnol'd, {\it Topological invariants of plane curves and caustics}, AMS Univ. Lect. Ser. 5, Providence, 1994.

\bibitem{Arnold95}
V.I. Arnol'd, {\it The geometry of spherical curves and the algebra of quaternions}, Russian Math. Surveys {\bf 50.1} (1995) 1.

\bibitem{B1874}
L. Bianchi, \emph{Sulle superficie a curvatura costante negativa}, Ann. Mat. Pura Appl. Serie 2, \textbf{7} (1874) 55--77.

\bibitem{BW91}
J.W. Bruce and T.C.Wilkinson, \textit{Folding maps and focal sets}, Singularity theory and its applications,
Part I (Coventry, 1988/1989), Lecture Notes in Math., vol. 1462, Springer, Berlin, 1991, pp. 63–72.

\bibitem{BT23}
J.W. Bruce and F.~Tari, \textit{On the affine geometry of congruences of lines}, ArXiv Preprint 2307.02311 (2023)

\bibitem{CR78}
T.E. Cecil and P.J. Ryan,
\textit{Focal sets of submanifolds}, Pacific J. Math. {\bf 78} (1978) 27--39.

\bibitem{DoCarmo92}
M.P. do Carmo, \emph{Riemannian Geometry}, Mathematics: Theory \& Applications, Birkhäuser Boston, Inc., Boston, MA, 1992. Translated from the second Portuguese edition by Francis Flaherty.

\bibitem{Dom}
P. Dombrowski, {\it  On the geometry of the tangent bundle}, J. Reine Angew. Math. {\bf 210} (1962) 73--88.

\bibitem{Dun}
M. Dunajski, {\it Oriented straight lines and twistor correspondence}, Geom. Dedicata {\bf 112} (2005), 239--247.


\bibitem{GG22}
N. Georgiou and B. Guilfoyle, {\it A new geometric structure on tangent bundles}, J. Geom. Phys. {\bf 172} (2022) 104415.

\bibitem{Gorkavvy2012}
V. Gorkavyy, {\it An Example of Bianchi transformation in $E^4$}, J. Math. Phys. Anal. Geom.  \textbf{8.3} (2012) 240--247.

\bibitem{Gorkavvy2015}
V. Gorkavyy, {\it Generalization of the Bianchi–B\"acklund transformation of pseudo-spherical surfaces}, J. Math. Sci. (N.Y.) \textbf{207.3} (2015) 467--484.

\bibitem{GK}
 B. Guilfoyle and W. Klingenberg, {\it An indefinite K\"ahler metric on the space of oriented lines}, J. London Math. Soc. {\bf 72} (2005) 497--509.

\bibitem{GK2}
B. Guilfoyle and W. Klingenberg, {\it A global version of a classical result of Joachimsthal}, Houston J. Math. {\bf 45} (2019) 455--467.

\bibitem{GK3}
B. Guilfoyle and W. Klingenberg, {\it  A neutral K\"ahler surface with applications in geometric optics}, European Mathematical Society Publishing House, Zurich (2008) 149--178.

\bibitem{Har16}
M.Harrison, \emph{Skew flat fibrations}, Math. Z. \textbf{282} (2016), 203–221. 

\bibitem{Huygens1673}
C. Huygens, \textit{Horologium oscillatorium}, 1673.

\bibitem{Kow}
O. Kowalski, {\it Curvature of the induced Riemannian metric on the tangent bundle of a Riemannian manifold }, J. Reine Angew. Math. {\bf 250} (1971) 124--129.


 \bibitem{Lee18}
J.M. Lee, {\it Introduction to Riemannian Manifolds}, 2nd edition, Graduate Texts in Mathematics {\bf 176}, Springer, Cham (2018).

\bibitem{LP20}
R. López and A. Pámpano, {\it Rotational surfaces of constant astigmatism in space forms}, 
J. Math. Anal. Appl. {\bf 483} (2020), 123602.


\bibitem{Port71}
I.R. Porteous, \textit{The normal singularities of a submanifold}, J. Differential Geom. \textbf{5} (1971).

\bibitem{Port94}
I.R. Porteous, \emph{Geometric Differentiation: For the Intelligence of Curvature}, Cambridge University Press, Cambridge, 1994.

\bibitem{RaS82}
C. Rogers and W.F. Shadwick, B\"acklund transformations and their applications, Academic Press (1982).




\bibitem{S05}
M. Salvai, \emph{On the geometry of the space of oriented lines of Euclidean space}, Manuscripta Math. \textbf{118} (2005) 181–-189. 


\bibitem{TaT79}
K. Tenenblat and C.-L. Terng, {\it A higher dimension generalization of the sine-Gordon equation and its B\"acklund transformation}, Bull. Amer. Math. Soc. (N.S.) \textbf{1.3} (1979) 589--593.

\bibitem{Terng1980}
C.-L. Terng, {\it A higher dimension generalization of the sine-Gordon equation and its soliton theory}, Ann. of Math.  \textbf{111.3} (1980) 491--510.





\bibitem{Wilczynski16}
E.J. Wilczynski, {\it Some remarks on the historical development and the future prospects of the differential geometry of plane curves}, Bull. Amer. Math. Soc. \textbf{22.7} (1916) 317-329.

\end{thebibliography}
\end{document}